\documentclass[12pt,a4paper,notitlepage]{article}
\usepackage{fullpage, amsfonts, amsthm, amsmath, setspace, graphicx, amssymb, float, colonequals, float,mathrsfs}
\usepackage[usenames,dvipsnames]{xcolor}

\newtheorem{theorem}{Theorem}[section]
\newtheorem{lemma}[theorem]{Lemma}
\newtheorem{fact}[theorem]{Fact}
\newtheorem{proposition}[theorem]{Proposition}

\newtheorem{case}{Case}[theorem]

\numberwithin{subcase}{case}
\numberwithin{subsubcase}{subcase}
\numberwithin{claim}{theorem}

\newenvironment{definition}[1][Definition]{\begin{trivlist}
\item[\hskip \labelsep {\bfseries #1}]}{\end{trivlist}}

\newenvironment{corollary-mannum}[2][Corollary]{\begin{trivlist}
\item[\hskip \labelsep {\bfseries #1}\hskip \labelsep {\bfseries #2}]}{\end{trivlist}}
\newenvironment{conjecture-mannum}[2][Conjecture]{\begin{trivlist}
\item[\hskip \labelsep {\bfseries #1}\hskip \labelsep {\bfseries #2}]}{\end{trivlist}}

\newcommand{\ds}{\displaystyle}

\title{\bf Boundary Convergence and Path Divergence Sets for Bounded Analytic Functions on the Disk}
\author{Trevor J. Richards}

\newcommand{\Addresses}{{
  \bigskip
  \footnotesize

  \noindent \textbf{T. J. Richards}\\
  \textsc{Department of Mathematics, Statistics, and Computer Science, Monmouth College\\
  Center for Science and Business Room 348\\
  Monmouth, Il 61462}\\
  \textit{E-mail address:} \texttt{trichards@monmouthcollege.edu}
}}

\begin{document}

\maketitle
\begin{center}\textbf{MSC 2010:} 30D40
\end{center}

\begin{abstract}
Let $f:\mathbb{D}\to\mathbb{C}$ be a bounded analytic function.  A set $K\subset\mathbb{D}$ which contains the point $1$ in its boundary is called a \textit{convergence set} for $f$ at $1$ if $f(z)$ converges to some value $\zeta$ as $z\to1$ with $z\in K$.  $K$ is called a \textit{path divergence set} for $f$ at $1$ if $f$ diverges along every path $\gamma$ which lies in $K$ and approaches $1$.  In this article, we show that for a path $\gamma$ through the unit disk from $-1$ to $1$, if $f$ fails to converge along $\gamma$, then either the region above $\gamma$ or the region below $\gamma$ is a path divergence set for $f$.  On the other hand, if $\gamma_1$ and $\gamma_2$ are two such paths, and $f$ converges along both $\gamma_1$ and $\gamma_2$, then the region between $\gamma_1$ and $\gamma_2$ is a convergence set for $f$.  This latter fact is immediate when $\gamma_1$ and $\gamma_2$ do not intersect except at their end-points, but becomes non-trivial when $\gamma_1$ and $\gamma_2$ are highly intersecting.  We conclude the paper with an examination of the convergence sets for the function $e^{\frac{z+1}{z-1}}$ at $1$.
\end{abstract}
\section{Introduction}\label{sect: Introduction.}%

Let $f:\mathbb{D}\to\mathbb{C}$ be analytic and bounded.  It is well known (see for example Theorem~5.2 in~\cite{C2}) that the non-tangential limit of $f$ exists at almost every point in $\partial\mathbb{D}$.  The sectorial limit theorem (see for example Theorem~5.4 in~\cite{C2}) states that if $f$ has a limit $\zeta$ along any path $\gamma$ in $\mathbb{D}$ which approaches some point $w\in\partial\mathbb{D}$, then the non-tangential limit of $f$ at $w$ exists and equals $\zeta$.  This article is concerned with local convergence properties of $f$ at individual points in $\partial\mathbb{D}$, so throughout we will assume that the point $w$ being approached is just the point $1$.  We will also assume for the sake of convenience that every path $\gamma$ mentioned in this article is a path in $\mathbb{D}$ from $-1$ to $1$.  That is, $\gamma:[0,1]\to\mathbb{C}$ with $\gamma(0)=-1$, $\gamma(1)=1$, and for all $s\in(0,1)$, $\gamma(s)\in\mathbb{D}$.  We will also use the symbol ``$\gamma$'' at times to denote the trace of $\gamma$ as a subset of the plane.

To say that $f$ has the non-tangential limit $\zeta$ at $1$ is to say that for any Stolz region $\mathcal{S}$ with vertex at $1$, $$\ds\lim_{z\to1,z\in\mathcal{S}}f(z)=\zeta.$$  In this article we generalize the notion of convergence or divergence in a Stolz region as follows.

\begin{definition}
Let $K\subset\mathbb{D}$ be a set, and assume that $1\in\partial K$.

\begin{itemize}
    \item If $\ds\lim_{z\to1,z\in K}f(z)$ exists, then we call $K$ a \textit{convergence set} for $f$ at $1$.
    \item If $\ds\lim_{s\to1^-}f(\gamma(s))$ does not exist for every path $\gamma\subset K$, then we call $K$ a \textit{path divergence set} for $f$ at $1$.
\end{itemize}

\end{definition}

We can now restate the sectorial limit theorem by saying that if there is a path $\gamma$ along which $f$ has a limit, then every Stolz region with vertex at $1$ is a convergence set for $f$ at $1$.  In Section~\ref{sect: Proof.}, we will prove the following theorem which guarantees the existence of convergence and path divergence sets for $f$ at $1$ in relation to paths along which $f$ converges or diverges respectively.  We must first make a definition.

\begin{definition}
Let $\gamma_1$, $\gamma_2$, and $\gamma_3$ be paths.

\begin{itemize}
\item By the \textit{region between $\gamma_1$ and $\gamma_2$} we mean the set of all points $z\in\mathbb{D}$ such that $z$ is in $\gamma_1$, $z$ is in $\gamma_2$, or $z$ is in a bounded component of $(\gamma_1\cup\gamma_2)^c$.
\item The component of $\mathbb{D}\setminus\gamma_3$ which contains $i$ in its boundary is called the \textit{region above $\gamma_3$}, and the component of $\mathbb{D}\setminus\gamma_3$ which contains $-i$ in its boundary is called the \textit{region below $\gamma_3$}.
\end{itemize}
\end{definition}

\begin{theorem}\label{thm: Domains from paths.}
Let $\gamma_1$, $\gamma_2$, and $\gamma_3$ be paths such that the limit of $f$ exists along $\gamma_1$ and along $\gamma_2$, and does not exist along $\gamma_3$.  Then the following hold.
\begin{itemize}
\item The region between $\gamma_1$ and $\gamma_2$ is a convergence set for $f$ at $1$.
\item Either the region above $\gamma_3$ or the region below $\gamma_3$ is a path divergence set for $f$ at $1$.
\end{itemize}
\end{theorem}

In Section~\ref{sect: Example.}, we will examine the convergence sets for the function $g(z)=e^{\frac{z+1}{z-1}}$ at $1$.  Although no open disk $D\subset\mathbb{D}$ which is tangent to the unit circle at $1$ can be a convergence set for $g$ at $1$, we will find a convergence set for $g$ at $1$ with smooth boundary close to $1$ that is tangent to the unit circle at $1$.
\section{Lemmas and the Proof of Theorem~\ref{thm: Domains from paths.}}\label{sect: Proof.}

In our proof of Theorem~\ref{thm: Domains from paths.}, we will use a result having to do with the cluster sets of a disk function, so we make the following definition.

\begin{definition}
Let $h:\mathbb{D}\to\hat{\mathbb{C}}$ be a meromorphic function such that $h$ extends continuously to each point in $\partial\mathbb{D}$ in a neighborhood of $1$ except possibly at $1$.

\begin{itemize}
\item The \textit{cluster set of $h$ at $1$}, denoted $\mathcal{C}(h,1)$, is the set of all values $\zeta\in\hat{\mathbb{C}}$ such that there is some sequence of points $\{z_n\}$ contained in the disk such that $z_n\to1$ and $h(z_n)\to\zeta$.
\item The \textit{boundary cluster set of $h$ at $1$}, denoted $\mathcal{C}_B(h,1)$, is the set of all values $\zeta\in\hat{\mathbb{C}}$ such that there is some sequence of points $\{z_n\}$ contained in the unit circle such that $z_n\to1$ and $h(z_n)\to\zeta$.
\end{itemize}
\end{definition}

We will use the following fact regarding these cluster sets (which appears as Theorem~5.2 in~\cite{CL}), which says that the boundary of the cluster set for a meromorphic disk function $h$ at $1$ is contained in the boundary cluster set of $h$ at $1$.

\begin{fact}\label{fact: Boundary convergence in disk and circle.}
Let $h:\mathbb{D}\to\hat{\mathbb{C}}$ be a meromorphic function such that $h$ extends continuously to each point in $\partial\mathbb{D}$ in a neighborhood of $1$ except possibly at $1$.  Then $$\partial\mathcal{C}(h,1)\subset\mathcal{C}_B(h,1).$$
\end{fact}

We continue with several lemmas.



\begin{lemma}\label{lem: Convergence of preimage.}
Let $\gamma$ be a path.  Then $\ds\lim_{z\in\gamma,z\to1}\gamma^{-1}(z)=1$, in the sense that as $z\in\gamma$ approaches $1$, $\max(1-s:s\in\gamma^{-1}(z))\to0$.
\end{lemma}

\begin{proof}
This follows from basic compactness and continuity properties.
\end{proof}

\begin{lemma}\label{lem: Path faces.}
Let $E_1$ and $E_2$ be disjoint circles in $\mathbb{C}$, and let $\gamma:[0,1]\to\mathbb{C}$ be any path.  Then only finitely many of the components of $\gamma^c$ can intersect both $E_1$ and $E_2$.
\end{lemma}

\begin{proof}
This fact follows from the uniform continuity of the path $\gamma:[0,1]\to\mathbb{C}$.
\end{proof}

\begin{lemma}\label{lem: Faces of gamma.}
Let $\gamma:[0,1]\to\mathbb{C}$ be a path, and let $F$ be a bounded component of $\gamma^c$.  Then for any Riemann map $\tau:\mathbb{T}\to F$, $\tau$ extends to a homeomorphism $\tau:cl(\mathbb{D})\to cl(F)$.
\end{lemma}

\begin{proof}
Let $E$ denote the union of every component of $\gamma^c$ other than $F$.  Since $E$ is open, $E$ is locally connected.  By the Hahn--Mazurkiewicz theorem (see for example Theorem 3-30 in~\cite{HY}) since $\gamma$ is the continuous image of the unit interval, $\gamma$ is locally connected.  Thus $\mathbb{C}\setminus F$ is the disjoint union of two locally connected sets, and is thus locally connected.

Now the Caratheodory--Torhorst theorem (see for example Theorem~5.5 in~\cite{C2}) gives that any Riemann map $\tau:\mathbb{D}\to F$ extends to a homeomorphism from the closure of the unit disk to the closure of $F$.
\end{proof}

\begin{lemma}\label{lem: Convergence in path.}
Let $\gamma$ be a path.  If $\ds\lim_{s\to1^-}f(\gamma(s))=0$, then $\ds\lim_{z\to1,z\in\gamma}f(z)=0$.
\end{lemma}

\begin{proof}
Fix an $\epsilon>0$, and choose a $\delta>0$ small enough that for every $s\in(1-\delta,1)$, $|f(\gamma(s))|<\epsilon$.  Define $\iota=\min_{s\in[0,1-\delta]}(|\gamma(s)-1|)$.  Then for any $z\in\gamma$ with $|z-1|<\iota$, we must have $\gamma^{-1}(z)\subset(1-\delta,1)$, and thus $|f(z)|<\epsilon$.  We conclude that $\ds\lim_{z\to1,z\in\gamma}f(z)=0$.
\end{proof}

\begin{proof}[Proof of Theorem~\ref{thm: Domains from paths.}.]
Let $\gamma_1,\gamma_2,\gamma_3:[0,1]\to\mathbb{C}$ be paths such that for each $i$, $\gamma_i(0)=-1$, $\gamma_i(1)=1$, and for each $s\in(0,1)$, $\gamma_i(s)\in\mathbb{D}$.

Let $\{z_n\}_{n=1}^\infty$ be a sequence of points contained in the region between $\gamma_1$ and $\gamma_2$ (this set is defined in the Introduction).  The sectorial limit theorem immediately implies that the limit of $f$ along both $\gamma_1$ and $\gamma_2$ is some common value.  Subtracting this value from $f$, we assume that this common limit equals $0$.  We wish to show that $f(z_n)\to0$ as $n\to\infty$.  Partition the natural numbers two sets $A$ and $B$ by the rule that $n\in A$ if $z_n$ is in either $\gamma_1$ or $\gamma_2$, and $n\in B$ if $z_n$ is contained in a bounded component of $(\gamma_1\cup\gamma_2)^c$.  It is immediately clear from Lemma~\ref{lem: Convergence of preimage.} that $\ds\lim_{n\to\infty,n\in A}f(z_n)=0$.  It remains to show that $\ds\lim_{n\to\infty,n\in B}f(z_n)=0$.

Fix an $\epsilon>0$.  Choose an $\iota_1>0$ such that for all $s\in(1-\iota_1,1)$, $|f(\gamma_1(s))|<\epsilon$ and $|f(\gamma_2(s))|<\epsilon$.  Define $\delta_1=\min_{s\in[0,1-\iota_1]}(|\gamma_1(s)-1|,|\gamma_2(s)-1|)$.  By Lemma~\ref{lem: Path faces.}, there are only finitely many components of $(\gamma_1\cup\gamma_2)^c$ that intersect both the domain $\{z:|z-1|>\delta_1\}$ and the disk $\{z:|z-1|<\delta_1/2\}$.  We wish to avoid those components of $(\gamma_1\cup\gamma_2)^c$ which intersect $\{z:\{|z-1|>\delta_1$, but which do not contain the point $1$ in their boundary.  Therefore we choose a $\delta_2\in(0,\delta_1/2)$ small enough so that every bounded component of $(\gamma_1\cup\gamma_2)^c$ which intersects the disk $\{z:|z-1|<\delta_2\}$ is either contained in the disk $\{z:|z-1|<\delta_1\}$, or contains $1$ in its boundary.

Let $F$ be some component of $(\gamma_1\cup\gamma_2)^c$ which intersects $\{z:|z-1|<\delta_2\}$.

\begin{case}
$F$ is contained in the disk $\{z:|z-1|<\delta_1\}$.
\end{case}

By choice of $\delta_1$, for all $z\in\partial F\cap\mathbb{D}$, $|f(z)|<\epsilon$.  Compactness and continuity considerations show that if $\partial F\cap\partial\mathbb{D}\neq\emptyset$, then $\partial F\cap\partial\mathbb{D}=\{1\}$, and work in the previous case shows that $\ds\lim_{z\to1,z\in F}f(z)=0$.  Thus we have that for all $w\in\partial F$, either $|f(w)|<\epsilon$ or $\ds\lim_{z\to w,z\in F}f(z)=0$.  The maximum modulus principle now implies that for all $z\in F$, $|f(z)|<\epsilon$.

\begin{case}\label{case: F is not contained in iota_1 ball.}
$F$ is not contained in the disk $\{z:|z-1|<\delta_1\}$.
\end{case}

By choice of $\delta_2$, the point $1$ is contained in $\partial F$.  Let $\tau:\mathbb{D}\to F$ be some Riemann map for $F$.  Lemma~\ref{lem: Faces of gamma.} implies that $\tau$ extends to a homeomorphism $\tau:cl(\mathbb{D})\to cl(F)$.  Thus we may adopt the normalization $\tau(1)=1$.  Since $f$ is analytic on $\partial F\setminus\{1\}$, it thus follows that $f\circ\tau$ extends continuously to every point on $\partial\mathbb{D}$ except possibly to $1$.  Moreover, as $\theta\to0$, $\tau(e^{i\theta})$ approaches $1$ in $\partial F$ (which is in turn contained in $\gamma_1\cup\gamma_2$), so that $f\circ\tau(e^{i\theta})$ approaches $0$ (by Lemma~\ref{lem: Convergence in path.}).

Therefore $\mathcal{C}_B(f\circ\tau,1)$ consists of the single point $0$ only.  Since $f\circ\tau$ is bounded in the disk, compactness considerations imply that $\mathcal{C}(f\circ\tau,1)$ is non-empty.   Fact~\ref{fact: Boundary convergence in disk and circle.} now immediately implies that $\mathcal{C}(f\circ\tau,1)=\{0\}$, and thus that $f\circ\tau(w)\to0$ as $w\to0$ in the disk.  Finally we conclude that $f(z)\to0$ as $z\to1$ in $F$.

Since $f(z)$ converges to $0$ as $z$ approaches $1$ in each of the finitely many components of $(\gamma_1\cup\gamma_2)^c$ which intersect both $B(1;\delta_1)^c$ and $B(1;\delta_2)$, and $|f(z)|<\epsilon$ for all $z$ in the components of $(\gamma_1\cup\gamma_2)^c$ which are contained entirely in $B(1;\delta_1)$, we can choose an $\delta_3\in(0,\delta_2)$ small enough so that for all $z$ contained in a bounded face of $(\gamma_1\cup\gamma_2)^c$, if $|z-1|<\delta_3$, then $|f(z)|<\epsilon$.  We conclude that $\ds\lim_{n\to\infty,n\in B}f(z_n)=0$, concluding the proof of the first item of the theorem.

In order to prove the second item, suppose by way of contradiction that neither the region above $\gamma_3$ nor the region below $\gamma_3$ is a domain of path divergence.  Then there are two paths $\psi_1$ and $\psi_2$ along which $f$ converges, such that $\psi_1$ lies above $\gamma_3$ and $\psi_2$ lies below $\gamma_3$.  By the first item of the theorem, the region between $\psi_1$ and $\psi_2$ is a convergence set for $f$ at $1$, but $\gamma_3$ lies in the region between $\psi_1$ and $\psi_2$, providing us with the desired contradiction.

\end{proof}

\section{Convergence Sets for $g(z)=e^{\frac{z+1}{z-1}}$}\label{sect: Example.}%

In this section we will explore the convergence sets of the function $g(z)=e^{\frac{z+1}{z-1}}$ at $1$.  The path divergence sets for $g$ at $1$ are also of interest, but as there do not appear to be concise geometric characterizations of these sets, we will restrict our attention to the convergence sets of $g$ at $1$.  We begin with a definition.

\begin{definition}\ 
\begin{itemize}
    \item For any $-\infty\leq p<q\leq\infty$, define $H_{Re}(p,q)=\{w\in\mathbb{C}:p<Re(w)<q\}$.
    \item For any $-\infty\leq p<q\leq\infty$, define $H_{Im}(p,q)=\{w\in\mathbb{C}:p<Im(w)<q\}$.
\end{itemize}
\end{definition}

Let $R:\hat{\mathbb{C}}\to\hat{\mathbb{C}}$ denote the M\"obius transformation $R(z)=\frac{z+1}{z-1}$.  $R$ is a conformal map from the unit disk $\mathbb{D}$ to the half plane $H_{Re}(-\infty,0)$.  Moreover if $\{z_n\}$ is any sequence of points in $\mathbb{D}$, $z_n\to 1$ if and only if $w_n=R(z_n)\to\infty$ in $H_{Re}(-\infty,0)$.  Thus in order to study the convergence sets of $g$ at $1$, it suffices to study the convergence sets in $H_{Re}(-\infty,0)$ of the function $h(w)=e^w$ at $\infty$, and we will treat the two settings as interchangable in what follows.

It is easy to see that as $w\to\infty$ in $H_{Re}(-\infty,0)$ in the real line, $h(w)\to0$.  Therefore by the sectorial limit theorem, if $\psi$ is any path in $H_{Re}(-\infty,0)$ approaching $\infty$, and $h$ converges along $\psi$, then $h$ must converge to $0$ along $\psi$ (and therefore back in the disk, $g(z)$ converges to $0$ as $z\to1$ in any Stolz region with vertex at $1$).

It is possible to construct a convergence set for $h$ at $\infty$ in which $h$ approaches any fixed value $\zeta\in\mathbb{D}$.  Indeed, for any such $\zeta$, there is a sequence of points in the $H_{Re}(-\infty,0)$ (namely $\{w_n=\ln(|\zeta|)+i(\operatorname{Arg}(\zeta)+2\pi n\}$) such that $w_n\to\infty$ and $h(w_n)=\zeta$.  This sequence is itself a convergence set for $h$ at $\infty$.  However since $g(z)\to0$ as $z\to1$ in any Stolz region in the unit disk with vertex at $1$, it is of more interest to us to look for large regions $G\subset H_{Re}(-\infty,0)$ (equivalently $K=R^{-1}(G)\subset\mathbb{D}$) such that $h(w)\to0$ as $w\to\infty$ in $G$ (equivalently $g(z)\to0$ as $z\to1$ in $K$).

Note that for $w=x+iy\in H_{Re}(-\infty,0)$, $|h(w)|=e^x$.  Therefore $h(w)\to0$ in $H_{Re}(-\infty,0)$ if any only if $x\to-\infty$.  Thus an unbounded set $G\subset H_{Re}(-\infty,0)$ is a convergence set for $h$ at $\infty$ in which $h$ approaches $0$ if and only if, for every number $p\in(-\infty,0)$, the set $G\cap H_{Re}(p,0)$ is bounded.  Since the region to the left of a verticle line in $H_{Re}(-\infty,0)$ is mapped by $R^{-1}$ to the interior of a circle contained in the unit disk which is tangent to the unit circle at $1$, we may translate the above observation back to the disk by saying that a set $K\subset\mathbb{D}$ containing $1$ in its boundary is a convergence set for $g$ in which $g$ approaches $0$ if and only if, for any disk $D$ contained in $\mathbb{D}$ and tangent to the unit circle at $1$, $K$ is eventually contained in $D$.

\begin{proposition}\label{prop: Convergence set for g.}
Let $K\subset\mathbb{D}$ be a set with $1\in\partial K$.  $K$ is a convergence set for $g$ at $1$ in which $g$ approaches $0$ if and only for every disk $D\subset\mathbb{D}$ which is tangent to the unit circle at $1$, there is an $\epsilon>0$ such that $K\cap B(1;\epsilon)\subset D$.
\end{proposition}

Note that by the same reasoning as above, for any $p\in(-\infty,0)$, $H_{Re}(p,0)$ is itself a path divergence set for $h$ at $\infty$, and thus for any disk $D$ strictly contained in $\mathbb{D}$ which is tangent to the unit circle at $1$, the region $\mathbb{D}\setminus D$ is a path divergence set for $g$ at $1$.

While no single disk contained in $\mathbb{D}$ and tanget to the unit circle at $1$ is a convergence set for $g$ at $1$, in the next subsection we will give a concrete example of a region $K\subset\mathbb{D}$ meeting the requirements of Proposition~\ref{prop: Convergence set for g.} which has smooth boundary at $1$ and is tangent to the unit circle at $1$.

\subsection{Convergence Set for $g$ at $1$ with Verticle Tangent Line at $1$}

We will finish by considering a specific convergence set for $g$ at $1$.  First let us translate the problem to the upper half plane $H_{Im}(0,\infty)$ via the M\"obius transformation $S:\hat{\mathbb{C}}\to\hat{\mathbb{C}}$ defined by $S(z)=\frac{-iz+i}{z+1}$.

$S$ transforms any disk contained in $\mathbb{D}$ which is tangent to the unit circle at $1$ to a disk contained in $H_{Im}(0,\infty)$ which is tangent to the real line at $0$.  Let $G\subset H_{Im}(0,\infty)$ be defined by $$G=\left\{x+iy:y>|x|^\frac{3}{2}\right\}.$$

$G$ is the region above the graph $y=T(x)=|x|^\frac{3}{2}$.  $T(0)=0$ and $T'(0)=0$, so $y=T(x)$ has the real line as its tangent line at $x=0$.  Fix for the moment some small $s>0$, and compare $y=T(x)$ to the graph $y=U_s(x)=s-\sqrt{s^2-x^2}$ (ie. the bottom half of the circle with radius $s$ centered at the point $(0,s)$).  An easy calculation shows that for sufficiently small positive $x$-values, $T'(x)>{U_s}'(x)$, so that the graph $y=T(x)$ lies above the half circle.

\begin{figure}[h]
  \centering
  \includegraphics[width=0.9\textwidth]{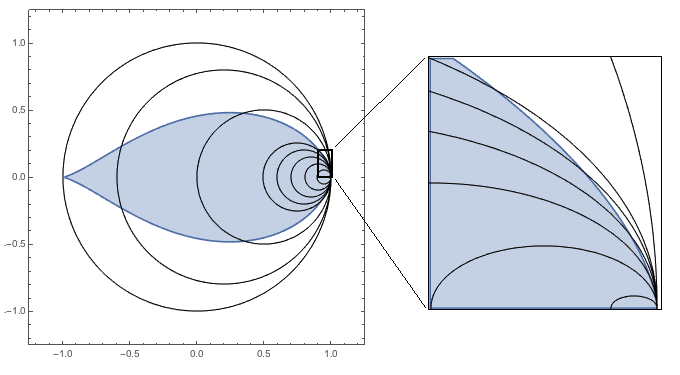}
  \caption{The convergence set $K$.}
  \label{fig: Convergence set for g.}
\end{figure}

Therefore defining $K=S^{-1}(G)$ (depicted in Figure~\ref{fig: Convergence set for g.}), we now have that $K$ is tangent to the unit circle at $1$, and for any circle $C$ in the disk which is tangent to the unit circle at $1$, the restriction of $K$ to some small neighborhood of $1$ is contained in the bounded face of $C$.  By Proposition~\ref{prop: Convergence set for g.}, $K$ is a domain of convergence for $g$ at $1$ in which $g$ approaches $0$.

\bibliographystyle{plain}
\bibliography{refs}

\Addresses

\end{document}